\documentclass{article}

\usepackage{microtype}
\usepackage{graphicx}
\usepackage{subfigure}
\usepackage{booktabs} 
\usepackage{enumitem}
\usepackage{bbding}
\usepackage{hyperref}
\usepackage{multirow}
\usepackage[table]{xcolor}

\usepackage[accepted]{icml2023}

\usepackage{amsmath}
\usepackage{amssymb}
\usepackage{mathtools}
\usepackage{amsthm}

\usepackage{amsmath}\allowdisplaybreaks
\usepackage{amsfonts,bm}
\usepackage{amssymb}

\def\eqref#1{equation~(\ref{#1})}

\def\1{\bf{1}}

\def\vone{{\bf{1}}}

\def\fO{{\mathcal{O}}}

\def\fX{{\mathcal{X}}}

\def\sB{{\mathbb{B}}}

\def\BE{{\mathbb{E}}}

\def\BR{{\mathbb{R}}}
\def\BS{{\mathbb{S}}}

\usepackage{amsthm}
\usepackage{etoolbox}
\theoremstyle{plain}
\newtheorem{thm}{Theorem}
\AfterEndEnvironment{thm}{\noindent\ignorespaces}
\newtheorem{dfn}{Definition}
\AfterEndEnvironment{defn}{\noindent\ignorespaces}

\AfterEndEnvironment{exmp}{\noindent\ignorespaces}
\newtheorem{lem}{Lemma}
\AfterEndEnvironment{lem}{\noindent\ignorespaces}
\newtheorem{asm}{Assumption}
\AfterEndEnvironment{asm}{\noindent\ignorespaces}

\AfterEndEnvironment{remark}{\noindent\ignorespaces}
\newtheorem{cor}{Corollary}
\AfterEndEnvironment{cor}{\noindent\ignorespaces}
\newtheorem{prop}{Proposition}
\AfterEndEnvironment{prop}{\noindent\ignorespaces}

\makeatletter
\def\Ddots{\mathinner{\mkern1mu\raise\p@
\vbox{\kern7\p@\hbox{.}}\mkern2mu
\raise4\p@\hbox{.}\mkern2mu\raise7\p@\hbox{.}\mkern1mu}}
\makeatother

\makeatletter
\newcommand*{\rom}[1]{\expandafter\@slowromancap\romannumeral #1@}
\makeatother

\definecolor{bgcolor}{rgb}{0.93,0.99,1}

\usepackage[capitalize,noabbrev]{cleveref}

\usepackage[textsize=tiny]{todonotes}

\icmltitlerunning{Faster Gradient-Free Algorithms for Nonsmooth Nonconvex Stochastic Optimization}

\begin{document}

\twocolumn[
\icmltitle{Faster Gradient-Free Algorithms for Nonsmooth Nonconvex Stochastic Optimization}
 


\icmlsetsymbol{equal}{*}

\begin{icmlauthorlist}
\icmlauthor{Lesi Chen}{sds}
\icmlauthor{Jing Xu}{iiis}
\icmlauthor{Luo Luo}{sds}
\end{icmlauthorlist}

\icmlaffiliation{sds}{School of Data Science, Fudan University, Shanghai, China}
\icmlaffiliation{iiis}{Institute for Interdisciplinary Information Sciences, Tsinghua University, Beijing, China}

\icmlcorrespondingauthor{Luo Luo}{luoluo@fudan.edu.cn}

\icmlkeywords{Machine Learning, ICML}

\vskip 0.3in
]



\printAffiliationsAndNotice{}  

\begin{abstract}
We consider the optimization problem of the form $\min_{x \in \mathbb{R}^d} f(x) \triangleq \mathbb{E}_{\xi} [F(x; \xi)]$, where the component $F(x;\xi)$ is $L$-mean-squared Lipschitz but possibly nonconvex and nonsmooth.
The recently proposed gradient-free method requires at most $\mathcal{O}( L^4 d^{3/2} \epsilon^{-4} + \Delta L^3 d^{3/2}  \delta^{-1} \epsilon^{-4})$ stochastic zeroth-order oracle complexity to find a
$(\delta,\epsilon)$-Goldstein stationary point of objective function, where~$\Delta = f(x_0) - \inf_{x \in \mathbb{R}^d} f(x)$ and $x_0$ is the initial point of the algorithm. 
This paper proposes a more efficient algorithm using stochastic recursive gradient estimators, which improves the complexity to~$\mathcal{O}(L^3 d^{3/2}  \epsilon^{-3}+ \Delta L^2 d^{3/2} \delta^{-1} \epsilon^{-3})$.
\end{abstract}


\section{Introduction}

We study the stochastic optimization problem 
\begin{align} \label{prob:main}
    \min_{x \in \BR^d} f(x) \triangleq \mathbb{E}_{\xi} [F(x; \xi)],
\end{align}
where the stochastic component $F(x;\xi)$, indexed by random variable $\xi$, is possibly nonconvex and nonsmooth.  
We focus on tackling the problem with Lipschitz continuous objective, which arises in many popular applications including 
simulation optimization~\cite{hong2015discrete,nelson2010optimization},
deep neural networks~\cite{nair2010rectified,glorot2011deep,chen2017zoo,ye2018hessian}, statistical learning~\cite{fan2001variable,zhang2006gene,zhang2010nearly,mazumder2011sparsenet,zhang2010analysis},  reinforcement learning \citep{mania2018simple,suh2022differentiable,choromanski2018structured,jing2021asynchronous}, financial risk minimization~\cite{stadtler2008supply} and supply chain management~\cite{duffie2010dynamic}.

The Clarke subdifferential~\cite{clarke1990optimization} for locally Lipschitz continuous function is a natural extension of gradients for smooth function and subdifferentials for convex functions.
Unfortunately,  hard instances suggest that finding a (near)~$\epsilon$-stationary point in terms of Clarke subdifferential is computationally intractable~\cite{zhang2020complexity,kornowski2021oracle}. 
\citet{zhang2020complexity} addressed this issue by proposing the notion of $(\delta,\epsilon)$-Goldstein
stationarity as a valid criterion for non-asymptotic convergence analysis in nonsmooth nonconvex optimization, which considers the convex hull of all Clarke subdifferential at points in the~$\delta$-ball neighbourhood~\cite{goldstein1977optimization}.
They also proposed the stochastic interpolated normalized
gradient descent method (SINGD) for finding a $(\delta,\epsilon)$-Goldstein stationary point of Hadamard directionally differentiable objective, which has  a stochastic first-order oracle complexity of $\fO(\Delta L^3  \delta^{-1} \epsilon^{-4})$, where $L$ is the  Lipschitz constant of the objective, $\Delta = f(x_0) - \inf_{x \in \BR^d} f(x)$ and $x_0$ is the initial point of the algorithm. In the case where both exact function value and gradient oracle are available, \citet{davis2021gradient} proposed the perturbed stochastic interpolated normalized gradient descent (PINGD) to relax the Hadamard directionally differentiable assumption.
Later, \citet{tian2022finite} proposed the perturbed stochastic interpolated normalized gradient descent (PSINGD) for stochastic settings. 
Very recently, in a concurrent work \citet{cutkosky2023optimal} presented an algorithm with a stochastic first-order oracle complexity of $\fO(\Delta L^2 \delta^{-1} \epsilon^{-3} )$ via a novel ``Online-to-Non-convex Conversion'' (O2NCC), and they also showed the optimality of their method by constructing a matching lower complexity bound. Besides, there also exist many works focusing on the complexity analysis under the special case when the stochastic first-order oracle is noiseless \cite{tian2021hardness,kornowski2021oracle,jordan2022complexity,kornowski2022complexity,davis2021gradient,tian2022no}. Remarkably, the hard instance constructed by \citet{tian2022no} shows that there is no deterministic algorithm that can compute a $(\delta,\epsilon)$-Goldstein stationary point in finite-time within dimension-free complexity. Besides, \citet{tian2022no} also proves that any deterministic finite-time algorithm for computing a $(\delta,\epsilon)$-Goldstein stationary point cannot be general zero-respecting, which rules out many commonly used first and second-order algorithms.

For many real applications~\citep{hong2015discrete,nelson2010optimization,chen2017zoo,ye2018hessian,stadtler2008supply,duffie2010dynamic,mania2018simple}, accessing the first-order oracle may be extremely expensive or even impossible. 
The randomized smoothing method is a well-known approach to designing zeroth-order algorithms~~\cite{nesterov2017random,ghadimi2013stochastic}.
It approximates the first-order oracle using the difference of function values.
Most of the existing works on zeroth-order optimization focus on convex problems~\cite{nesterov2017random,shamir2017optimal,duchi2015optimal,bubeck2012towards,duchi2015optimal,shamir2013complexity,bach2016highly} and smooth nonconvex problems~\cite{nesterov2017random,ghadimi2013stochastic,liu2018zeroth,ji2019improved,fang2018spider}.
Recently, \citet{lin2022gradient} considered the nonsmooth nonconvex case by establishing the relationship between Goldstein subdifferential~\cite{goldstein1977optimization} and randomized smoothing~\cite{nesterov2017random,shamir2017optimal}. As a consequence, they showed the gradient-free method (GFM)
could find a $(\delta,\epsilon)$-Goldstein stationary point within at most~$\mathcal{O}( L^4 d^{3/2} \epsilon^{-4} + \Delta L^3 d^{3/2}  \delta^{-1} \epsilon^{-4})$ stochastic zeroth-order oracle calls.

\begin{table*}[t]
    \centering
    \caption{We present the complexities of finding a $(\delta,\epsilon)$-Goldstein stationary point for $d$-dimensional $L$-Lipschitz objective under both deterministic and stochastic settings, where we denote $\Delta=f(x_0)-f^*$ and $x_0$ is the initial point of the algorithm.
    We remark for ``*'' that the algorithms proposed by \citet{zhang2020complexity} use the non-standard first-order oracle called the Hadamard directional derivative. }\vskip0.1cm
    \label{tab:res}    
    \begin{tabular}{ccccc}
    \hline
    Setting & Oracle & Method    & Reference & {Complexity}   \\
    \hline \hline \addlinespace
    & ~~1st* & SINGD & \citet
    {zhang2020complexity} & $\tilde \fO\left(
    \dfrac{\Delta L^3 }{\delta \epsilon^4}\right)$   \\ \addlinespace
     & 1st  & PSINGD    & \citet{tian2022finite}  & $\tilde \fO\left(\dfrac{\Delta L^3 }{\delta \epsilon^4}\right)$  \\ \addlinespace
     Stochastic & 1st & O2NCC    & \citet{cutkosky2023optimal}  & $\fO\left(\dfrac{\Delta L^2 }{\delta \epsilon^3}\right)$  \\ \addlinespace
    &0th & GFM & \citet{lin2022gradient} &  $\fO\left(d^{3/2} \left(
    \dfrac{  L^4 }{\epsilon^4}
    + \dfrac{ \Delta L^3 }{\delta \epsilon^4}\right) \right)$  \\ \addlinespace
    \rowcolor{gray!20} & 0th & ~~GFM$^+$ &Theorem \ref{thm:GFM+} & $\fO\left(d^{3/2} \left(
    \dfrac{  L^3 }{\epsilon^3}
    + \dfrac{ \Delta L^2 }{\delta \epsilon^3}\right) \right)$ \\
    \addlinespace
    \hline
    \addlinespace
     & ~~0th \& 1st*  & INGD & \citet{zhang2020complexity} & $\tilde \fO\left(\dfrac{\Delta L^2 }{\delta \epsilon^3}\right)$   \\ \addlinespace
      Deterministic & 0th \& 1st  & PINGD    & \citet{davis2021gradient} & $\tilde \fO\left(\dfrac{\Delta L^2 }{\delta \epsilon^3}\right)$  \\
     \addlinespace
    & 0th \& 1st & cutting plane    & \citet{davis2021gradient}  & $\tilde \fO\left(\dfrac{ \Delta L d}{\delta \epsilon^2}\right)$   \\ \addlinespace
    \hline 
    \end{tabular}
\end{table*}

\begin{table*}[t]
    \centering
    \caption{We present the complexities of finding a $(\delta,\epsilon)$-Goldstein stationary point for $d$-dimensional $L$-Lipschitz objectives with zeroth-order oracles under the stochastic setting, where we denote  $R= {\rm dist} (x_0, \fX^*)$ and $x_0$ is the initial point of the algorithm.}\vskip0.25cm
    \label{tab:res-convex}    
    \begin{tabular}{ccc}
    \hline
    Method & Reference & Complexity \\
    \hline \hline \addlinespace
    GFM & \citet{lin2022gradient} & $\fO \left( \dfrac{d^{3/2} L^4}{\epsilon^4} + \dfrac{d^{3/2}~  L^4~ R~ }{\delta\epsilon^4}  \right)$ \\ \addlinespace
     WS-GFM & Theorem \ref{thm:2GFM}  & ~~~~~$\mathcal{O}\left( \dfrac{d^{3/2} L^4}{\epsilon^4} + {\dfrac{d^{4/3}  L^{8/3}  R^{2/3}}{\delta^{2/3} \epsilon^{4/3}}} \right)$ \\ \addlinespace
    ~~GFM${}^+$ & Theorem \ref{thm:GFM+} & $\fO\left( \dfrac{d^{3/2} L^3}{\epsilon^3} + \dfrac{d^{3/2}~ L^3~R~}{\delta \epsilon^3~}  \right)$ \\ \addlinespace
    ~~WS-GFM${}^+$ & Theorem \ref{thm:2GFM+} &  ~~~$\mathcal{O}\left( \dfrac{d^{3/2} L^3}{\epsilon^3} + \dfrac{d^{4/3}  L^2 R^{2/3} }{\delta^{2/3} \epsilon^2} \right)$ \\ \addlinespace
    \hline 
    \end{tabular}
\end{table*}

In this paper, we propose an efficient stochastic gradient-free method named GFM$^+$ for nonsmooth nonconvex stochastic optimization. 
The algorithm takes advantage of randomized smoothing to construct stochastic recursive gradient estimators~\cite{nguyen2017sarah,fang2018spider,li2021page,li2019ssrgd} for the smoothed surrogate of the objective. 
It achieves a stochastic zeroth-order oracle complexity of $\fO( L^3 d^{3/2} \epsilon^{-3}+ \Delta L^2 d^{3/2} \delta^{-1} \epsilon^{-3})$ for finding a $(\delta,\epsilon)$-Goldstein stationary point, improving the dependency both on $L$ and~$\epsilon$ compared with GFM~\cite{lin2022gradient}.
In the case of $\delta L \lesssim \Delta$, the dominating term in the zeroth-order oracle complexity 
of GFM$^+$  becomes $\fO(d^{3/2} \Delta L^2 \delta^{-1} \epsilon^{-3} )$, which matches the optimal rate of the first-order method O2NCC \cite{cutkosky2023optimal} except an unavoidable additional dimensional factor due to we can only access the zeroth order oracles~\cite{duchi2015optimal}.
We summarize the results of this paper and related works in Table~\ref{tab:res}, where the deterministic setting refers to the case that both the exact zeroth-order and first-order oracle are available. 

We also extend our results to nonsmooth convex optimization.
The optimality of the zeroth-order algorithms to minimize convex function in the measure of function value has been established by \citet{shamir2017optimal}, while the goal of finding approximate stationary points is much more challenging \citep{allen2018make,nesterov2012make,lee2021geometric}. 
The lower bound provided by \citet{kornowski2022complexity} suggests finding a point with a small subgradient for Lipschitz convex objectives is intractable.
Hence, finding an approximate Goldstein stationary point is also a more reasonable task in convex optimization. We propose the two-phase gradient-free methods to take advantage of the convexity. It shows that GFM$^+$ with warm-start strategy could find a $(\delta,\epsilon)$-Goldstein stationary point within~$\mathcal{O}\left( L^3d^{3/2}\epsilon^{-3} + R L^2 d^{4/3} \delta^{-2/3} \epsilon^{-2}\right)$ stochastic zeroth-order oracle complexity, where $R$ is the distance of the initial point to the optimal solution set.
We summarize the results for convex case in Table \ref{tab:res-convex}.

\section{Preliminaries} \label{sec:pre}

In this section, we introduce the background for nonsmooth nonconvex function and randomized smoothing technique in zeroth-order optimization. 



\subsection{Notation and Assumptions}
Throughout this paper, we use  $\Vert \cdot \Vert$ to represent the Euclidean norm of a vector. 
$\sB_{\delta}(x)\triangleq \{y \in \BR^d: \Vert y-x \Vert \le \delta\}$ represents the Euclidean ball centered at point $x$ with radius~$\delta$. 
We also define $\sB = \sB_{1}(0)$ and $\BS \triangleq \{x \in \BR^d: \Vert x \Vert =1\}$ as the unit Euclidean ball and sphere centered at origin respectively. 
We let ${\rm conv} \{A\}$ be the convex hull of set~$A$ and ${\rm dist}(x,A) = \inf_{y \in A } \Vert x - y \Vert $ be the distance between vector $x$ and set $A$. When measuring complexity, we use the notion $\tilde \fO(\,\cdot\,)$ to hide  the logarithmic factors.

Following the standard setting of stochastic optimization, we assume the objective $f$ can be expressed as an expectation of some stochastic components. 

\begin{asm} \label{asm:stoc}
We suppose that the objective function has the form of~$f(x) = \BE_{\xi} [F(x;\xi)]$, where $\xi$ denotes the random index.
\end{asm}
We focus on the Lipschitz function with finite infimum, which satisfies the following assumptions.
\begin{asm} \label{asm:Lip}
We suppose that the stochastic component $F(\,\cdot\,;\xi):\BR^d\to\BR$ is $L(\xi)$-Lipschitz for any $\xi$, i.e. it holds that
\begin{align} \label{eq:eveLip}
    \Vert F(x;\xi) - F(y;\xi) \Vert \le L(\xi) \Vert x - y \Vert
\end{align}
for any $x,y \in \BR^d$ and $L(\xi)$ has bounded second-order moment, i.e. there exists some constant $L>0$ such that 
\begin{align*}
    \BE_{\xi} \big[L(\xi)^2\big] \le L^2.
\end{align*}
\end{asm}

\begin{asm} \label{asm:lower}
We suppose that  the objective $f:\BR^d\to\BR$ is lower bounded and define 
\begin{align*}
f^* := \inf_{x \in \BR^d}f(x).
\end{align*}
\end{asm}

The inequality (\ref{eq:eveLip}) in Assumption \ref{asm:Lip} implies the mean-squared Lipschitz continuity of $F(x;\xi)$. 
\begin{prop}[mean-squared continuity] \label{prop:msLip}
Under Assumption \ref{asm:Lip}, for any $x, y \in \BR^d$ it holds that
\begin{align*}
    \BE_{\xi} \Vert F(x;\xi) - F(y;\xi) \Vert^2 \le L^2 \Vert x - y \Vert^2.
\end{align*}
\end{prop}

All of above assumptions follow the setting of \citet{lin2022gradient}. We remark that Assumption \ref{asm:Lip} is weaker than assuming each stochastic component $F(\,\cdot\,;\xi)$ is $L$-Lipschitz.

\subsection{Goldstein Stationary Point}

According to Rademacher's theorem, a Lipschitz function is differentiable almost everywhere. Thus, we can define the Clarke subdifferential \citep{clarke1990optimization} as well as approximate Clarke stationary point as follows.

\begin{dfn}[Clarke subdifferential] \label{dfn:Clk}
The Clarke subdifferential of  a Lipschitz function $f$ at a point $x$ is defined by
\begin{align*}
    \partial f(x):=\operatorname{conv}\left\{ g: g = \lim_{x_k \rightarrow x} \nabla f(x_k)\right\}.
\end{align*}
\end{dfn}

\begin{dfn}[Approximate Clarke stationary point]
Given a Lipschitz function $f$, we say the point $x$ is an $\epsilon$-Clarke stationary point of $f$ if it holds that
\begin{align*}
    {\rm dist}(0, \partial f(x))  \leq \epsilon.
\end{align*}
\end{dfn}
However, finding an $\epsilon$-Clarke stationary point 
is computationally intractable~\citep{zhang2020complexity}. 
Furthermore, finding a point that is $\delta$-close to an $\epsilon$-stationary point may have an inevitable exponential dependence on the problem dimension~\citep{kornowski2021oracle}. 
As a relaxation, we pursue the $(\delta,\epsilon)$-Goldstein stationary point~\citep{zhang2020complexity}, whose definition is based on the following Goldstein subdifferential~\citep{goldstein1977optimization}. 

\begin{dfn}[Goldstein subdifferential] \label{dfn:Gold}
Given $\delta >0$, the Goldstein $\delta$-subdifferential of aLipschitz function $f$ at a point $x$ is defined by
\begin{align*}
    \partial_{\delta} f(x):=\operatorname{conv}\left\{\cup_{y \in \mathbb{B}_{\delta}(x)} \partial f(y)\right\},
\end{align*}
where $\partial f(x)$ is the Clarke subdifferential.
\end{dfn}

The $(\delta,\epsilon)$-Goldstein stationary point~\citep{zhang2020complexity} is formally defined as follows.
\begin{dfn}[Approximate Goldstein stationary point] \label{dfn:Goldsta}
Given a Lipschitz function $f$, we say the point $x$ is a $(\delta,\epsilon)$-Goldstein stationary point of $f$ if it holds that
\begin{align} \label{eq:delta-eps}
    {\rm dist}(0, \partial_{\delta} f(x))  \leq \epsilon.
\end{align}
\end{dfn}

Our goal is to design efficient stochastic algorithms to find a $(\delta,\epsilon)$-Goldstein stationary point in expectation.

\subsection{Randomized Smoothing}


Recently, \citet{lin2022gradient} established the relationship between uniform smoothing and Goldstein subdifferential. 
We first present the definition of the smoothed surrogate.

\begin{dfn}[uniform smoothing] \label{dfn:f-delta}
Given a Lipschitz function $f$, we denote its smoothed surrogate as
\begin{align*}
    f_{\delta}(x) := \BE_{u\sim\mathcal P}[f(x + \delta u)],
\end{align*} 
where $\mathcal P$ is the uniform distribution on unit ball $\sB$.
\end{dfn}

The smoothed surrogate has the following properties~\citep[Proposition 2.3 and Theorem 3.1]{lin2022gradient}.

\begin{prop}\label{prop:smooth}
Suppose that function $f:\BR^d\to\BR$ is $L$-Lipschitz, then it holds that:
\begin{itemize}
    \item  $\vert f_{\delta}(\,\cdot\,) - f(\,\cdot\,) \vert \le \delta L$.
    \item $f_{\delta}(\,\cdot\,)$ is $L$-Lipschitz.
    \item $f_{\delta}(\,\cdot\,)$ is differentiable and with $c \sqrt{d}L\delta^{-1}$-Lipschitz gradient for some numeric constant $c>0$. 
    \item  $\nabla f_{\delta}(\,\cdot\,) \in \partial_{\delta} f(\,\cdot\,)$, where $\partial_{\delta} f(\,\cdot\,)$ is the Goldstein subdifferentiable.
\end{itemize}
\end{prop}

\citet{flaxman2004online} showed that we can obtained an unbiased estimate of $\nabla f_{\delta}(x)$ by using the two function value evaluations of points sampled on unit sphere $\BS$, which leads to the zeroth-order gradient estimator.

\begin{dfn}[zeroth-order gradient estimator] \label{dfn:zo-est}
Given a stochastic component $F(\,\cdot\,;\xi):\BR^d\to\BR$, we denote its stochastic zeroth-order gradient estimator at $x\in\BR^d$ by:
\begin{align*}  
    g(x;w,\xi) = \frac{d}{2 \delta} (F(x+ \delta w ;\xi) - F(x-\delta w;\xi))w
\end{align*}
where $w \in \BR^d$ is sampled from a uniform distribution on a unit sphere $\BS$.
\end{dfn}

We also introduce the mini-batch zeroth-order gradient estimator that plays an important role in the stochastic zeroth-order algorithms.

\begin{dfn}[Mini-batch zeroth-order gradient estimator] \label{dfn:bgrad}
Let $S=\{(\xi_i,w_i)\}_{i=1}^b$, where vectors $w_1,\dots,w_{b}\in \BR^d$ are i.i.d. sampled from a uniform distribution on unit sphere $\BS$ and random indices $\xi_1,\dots,\xi_b$ are i.i.d.
We denote the mini-batch zeroth-order gradient estimator of $f(\,\cdot\,):\BR^d\to\BR$ in terms of $S$ at $x\in\BR^d$ by
\begin{align*} 
    g(x;S) = \frac{1}{b} \sum_{i=1}^b g(x; w_i,\xi_i).
\end{align*}
\end{dfn}

Next we present some properties for zeroth-order gradient estimators. 
\begin{prop}[{Lemma D.1 of \citet{lin2022gradient}}]\label{prop:zo}
Under Assumption \ref{asm:stoc} and \ref{asm:Lip}, it holds that
\begin{align*}
    \BE_{w,\xi} [g(x;w,\xi)] &= \nabla f_{\delta}(x)
\end{align*}
and
\begin{align*}    
    \mathbb{E}_{w,\xi} \Vert g(x;w,\xi) \Vert^2 &\le 16 \sqrt{2 \pi} dL^2.
\end{align*}
\end{prop}

\begin{cor} \label{cor:bzo}
Let $S=\{(\xi_i,w_i)\}_{i=1}^b$, then under Assumption \ref{asm:stoc} and \ref{asm:Lip} it holds that
\begin{align*}
     \mathbb{E}_{S} \Vert g(x;S) - \nabla f_{\delta}(x) \Vert^2 \le \frac{16 \sqrt{2 \pi} dL^2}{b} .
\end{align*}
\end{cor}

\noindent The following continuity condition of gradient estimator is an essential element for variance reduction.
\begin{prop} \label{prop:avgL}
Under Assumption \ref{asm:stoc} and \ref{asm:Lip}, for any $w \in \BS$ and $x,y \in \BR^d$, it holds that
\begin{align*}
    \BE_{\xi }\Vert g(x;w,\xi) - g(y;w,\xi) \Vert^2 \le \frac{d^2 L^2}{\delta^2}  \Vert x - y \Vert^2.
\end{align*}
\end{prop}

To simplify the presentation, we introduce the following notations:
\begin{align} \label{dfn:const}
\begin{split}
& L_{\delta} = \frac{c\sqrt{d}L}{\delta},~~ \sigma_{\delta}^2 = 16 \sqrt{2 \pi} dL^2,~~\Delta=f(x_0)-f^* \\
& M_{\delta} = \frac{d L}{\delta},~~~~f_{\delta}^* := \inf_{x \in \BR^d} f_{\delta}(x),~~\Delta_{\delta} = \Delta + L \delta,
\end{split}
\end{align}
where $x_0$ is the initial point of the algorithm.

Then the above results can be written as
\begin{itemize}
    \item $\Vert \nabla f_{\delta}(x) - \nabla f_{\delta}(y) \Vert \le L_{\delta} \Vert x - y \Vert$ for any $x,y \in \BR^d$, where $f_{\delta}$ follows Definition \ref{dfn:f-delta};
    \item $\BE_{\xi }\Vert g(x;w,\xi) - g(y;w,\xi) \Vert^2 \le M_{\delta}^2 \Vert x - y \Vert^2$ for any~$x,y\in\BR^d$, where $w$ and $\xi$ follow Definition~\ref{dfn:zo-est};    
    \item $\BE_{S} \Vert g(x;S) - \nabla f_{\delta}(x) \Vert^2 \le {\sigma_{\delta}^2}/{b}$ for any~$x \in \BR^d$, where $S$ and $b$ follow Definition \ref{dfn:bgrad};
    \item $f_{\delta}(x_0) - f_{\delta}^*  \le \Delta_{\delta}$.
\end{itemize}

In Appendix \ref{apx:tight}, we also prove the orders of Lipschitz constants $L_{\delta} = \fO(\sqrt{d} L \delta^{-1})$ and~$M_{\delta} = \fO(d L \delta^{-1})$ are tight in general.
However, it remains unknown whether the order of $\sigma_{\delta}$ could be improved.

\section{Algorithms and Main Results} 

This section introduces GFM$^+$ for nonsmooth nonconvex stochastic optimization problem~(\ref{prob:main}). We also provide complexity analysis to show the algorithm has better theoretical guarantee than GFM~\cite{lin2022gradient}. 


\subsection{The Algorithms}
\begin{algorithm}[t]  
\caption{GFM ($ x_0,\eta ,T $)} \label{alg:GFM}
\begin{algorithmic}[1] 
\STATE \textbf{for} $t = 0,1,\cdots,T-1$  \\[0.1cm]
\STATE \quad sample a random direction $w_t \in \BS$ and a random index $\xi_t$ \\[0.1cm]
\STATE \quad update $x_{t+1} = x_t - \eta g(x_t; w_t,\xi_t)$ \\[0.1cm]
\STATE \textbf{end for} \\[0.1cm]
\STATE \textbf{return} $x_{\rm out}$  chosen uniformly from $\{ x_t\}_{t=0}^{T-1}$ 
\end{algorithmic}
\end{algorithm}

\begin{algorithm}[t] 
\caption{GFM${}^+$($ x_0,\eta ,T, m , b,b'$)} \label{alg:GFM+}
\begin{algorithmic}[1] 
\STATE $v_0 = g(x_0;S')$ \\[0.1cm]
\STATE \textbf{for} $t = 0,1,\cdots,T-1$  \\[0.1cm]
\STATE \quad \textbf{if} $t~{\rm mod}~m =0$ \\[0.1cm]
\STATE \quad \quad sample $S' = \{(\xi_i,w_i) \}_{i=1}^{b'}$ \\[0.1cm]
\STATE \quad \quad calculate $v_t = g(x_t ;S') $ \\[0.1cm]
\STATE \quad \textbf{else} \\[0.1cm]
\STATE \quad \quad sample $S = \{(\xi_i,w_i) \}_{i=1}^b$   \\[0.1cm]
\STATE \quad \quad calculate $v_t = v_{t-1} + g(x_t;S) - g(x_{t-1};S)$ \\[0.1cm]
\STATE \quad \textbf{end if} \\[0.1cm]
\STATE \quad update $x_{t+1} = x_t - \eta v_t$ \\[0.1cm]
\STATE \textbf{end for} \\[0.1cm]
\STATE \textbf{return} $x_{\rm out}$  chosen uniformly from $\{ x_t\}_{t=0}^{T-1}$ 
\end{algorithmic}
\end{algorithm}

We propose GFM$^+$ in Algorithm~\ref{alg:GFM+}.
Different from GFM (Algorithm~\ref{alg:GFM})~\cite{lin2022gradient} that uses a vanilla zeroth-order gradient estimator~$g(x_t;w_t,\xi_t)$, GFM$^+$ approximates $\nabla f_\delta (x_t)$ by recursive gradient estimator $v_t$ with update rule
\begin{align*}
    v_{t} = v_{t-1} + g(x_t;S) - g(x_{t-1};S).
\end{align*}
The estimator $v_t$ properly reduces the variance in estimating~$\nabla f_\delta(x)$, which leads to a better stochastic zeroth-order oracle upper complexity than GFM. 

The variance reduction is widely used to design stochastic first-order and zeroth-order algorithms for smooth nonconvex optimization~\citep{fang2018spider,pham2020proxsarah,wang2019spiderboost,nguyen2017sarah,ji2019improved,huang2022accelerated,liu2018zeroth,levy2021storm+,cutkosky2019momentum}.
The existing variance-reduced algorithms for nonsmooth problem~\cite{pham2020proxsarah,j2016proximal,xiao2014proximal,li2022simple} require the objective function to have a composite structure, which does not include our setting that each stochastic component $F(x;\xi)$ could be nonsmooth.


\subsection{Complexity Analysis} \label{sec:GFM+}

This subsection considers the upper bound complexity of proposed GFM$^+$.
First of all, we present the descent lemma.
\begin{lem}[{\citet[Lemma 2]{li2021page}}] \label{lem:delem}
Under Assumption~\ref{asm:stoc} and~\ref{asm:Lip} , Algorithm \ref{alg:GFM+} holds that
\begin{align*}
f_{\delta}(x_{t+1}) &\le f_{\delta}(x_t) -  \frac{\eta}{2} \Vert \nabla f_{\delta}(x_t) \Vert^2 \\
&\quad +\frac{\eta}{2} \Vert v_t - \nabla f_{\delta}(x_t) \Vert^2-\left( \frac{\eta}{2} - \frac{L_{\delta} \eta^2}{2} \right) \Vert v_t \Vert^2, 
\end{align*}
where $L_{\delta}$ follows the definition in (\ref{dfn:const}).
\end{lem}


Secondly, we show the variance bound of stochastic recursive gradient estimators for smooth surrogate $f_\delta (x)$, which is similar to Lemma 1 of \citet{fang2018spider}.
\begin{lem}[Variance bound] \label{lem:vrb}
Under Assumption \ref{asm:stoc}, \ref{asm:Lip} , for Algorithm \ref{alg:GFM+} it holds that
\begin{align*}
&\quad \mathbb{E} \Vert v_{t+1} - \nabla f_{\delta}(x_{t+1}) \Vert^2 \\
&\le \mathbb{E} \left[ \Vert v_{t} - \nabla f_{\delta}(x_{t}) \Vert^2 + \frac{ \eta^2 M_{\delta}^2}{b} \Vert v_t \Vert^2 \right],
\end{align*}
where $M_{\delta}$ follows the definition in (\ref{dfn:const}).
\end{lem}

For convenience, we denote $n_t \triangleq \lceil t / m \rceil$ as the index of epoch such that $(n_t-1)m  \le t \le n_t m -1 $. 
Then Lemma \ref{lem:vrb} leads to the following corollary. 
\begin{cor} \label{cor:vrb}
Under Assumption \ref{asm:stoc} and \ref{asm:Lip}, for Algorithm \ref{alg:GFM+} it holds that
\begin{align*}
\mathbb{E} \Vert v_{t+1} - \nabla f_{\delta} (x_{t+1}) \Vert^2 
\le \frac{\sigma_{\delta}^2}{b'} + \frac{\eta^2 M_{\delta}^2}{b} \sum_{i = m(n_t-1)}^t \!\!\BE \Vert v_i \Vert^2.
\end{align*}
\end{cor}
Combing the descent Lemma \ref{lem:delem} and  Corollary \ref{cor:vrb}, we obtain the progress of one epoch.

\begin{lem}[One epoch progress] \label{lem:epoch}
Under Assumption \ref{asm:stoc}~and~\ref{asm:Lip}, for Algorithm \ref{alg:GFM+} it holds that
\begin{align*}
&\quad \mathbb{E}[ f_{\delta}(x_{m n_t})] \\ &\le  f_{\delta}(x_{ m(n_t-1)})  -  \frac{\eta}{2} \sum_{i=m(n_t-1)}^{m n_t-1} \BE \Vert \nabla f_{\delta}(x_i) \Vert^2 +  \frac{ m \eta  \sigma_{\delta}^2}{2  b' } \\
&\quad - \left( \frac{\eta}{2} - \frac{L_{\delta} \eta^2}{2} - \frac{m\eta^3 M_\delta^2}{2 b} \right) \sum_{i=m(n_t-1)}^{mn_t-1} \BE \Vert v_i \Vert^2,
\end{align*}
where $L_{\delta}, M_{\delta}, \sigma_{\delta}$ follow the definition in (\ref{dfn:const}).
\end{lem}
Connecting the progress for all of $T$ epochs leads to the following result.
\begin{cor} \label{cor:epoch}
Under Assumption \ref{asm:stoc} and \ref{asm:Lip} , for Algorithm~\ref{alg:GFM+} it holds that
\begin{align}\label{ieq:star}
\small\begin{split}
 \mathbb{E}[ f_{\delta}(x_{T})] &\le  f_{\delta}(x_0)  -  \frac{\eta}{2} \sum_{i=0}^{T-1} \BE \Vert \nabla f_{\delta}(x_i) \Vert^2 +  \frac{ T \eta  \sigma_{\delta}^2}{2  b' }\\
&\quad - \underbrace{\left( \frac{\eta}{2} - \frac{L_{\delta} \eta^2}{2} - \frac{m\eta^3 M_\delta^2}{2 b } \right)}_{(*)} \sum_{i=0}^{T-1} \BE \Vert v_i \Vert^2.
\end{split}
\end{align}
\end{cor}
Using Corollary \ref{cor:epoch} with 
\begin{align*}
    \eta = \frac{\sqrt{b'}}{m M_{\delta}},\quad m = \left\lceil\frac{L_{\delta} \sqrt{b'}}{M_{\delta}}\right\rceil \quad \text{and} \quad b = \left \lceil \frac{2b'}{m}\right \rceil,
\end{align*}
we know that the term~(*) in inequality (\ref{ieq:star}) is positive and obtain
\begin{align} \label{eq:main}
\small\begin{split}
\mathbb{E}[ f_{\delta}(x_{T})] \le  f_{\delta}(x_0)  -  \frac{\eta}{2} \sum_{i=0}^{T-1} \BE \Vert \nabla f_{\delta}(x_i) \Vert^2 +  \frac{ T \eta  \sigma_{\delta}^2}{2 b'},
\end{split}
\end{align}
which means 
\begin{align} \label{eq:station}
    \BE \Vert \nabla f_{\delta} (x_{\rm out} ) \Vert^2 \le \frac{2 \Delta{\delta}}{\eta T} + \frac{\sigma_{\delta}^2}{b'}. 
\end{align}
Applying inequality (\ref{eq:station}) with
\begin{align} 
~T= \left \lceil \frac{4 \Delta_{\delta}}{\eta \epsilon^2} \right \rceil \quad\text{and}\quad b' = \left \lceil  \frac{2 \sigma_{\delta}^2}{\epsilon^2} \right \rceil,
\end{align}
we conclude that the output $x_{\rm out}$ is an $\epsilon$-stationary point of the smooth surrogate $f_{\delta}(\,\cdot\,)$ in expectation.


Finally, the relationship $\nabla f_{\delta}(x_{\rm out}) \in \partial_{\delta} f(x_{\rm out})$  shown in Proposition \ref{prop:smooth} indicates that Algorithm~\ref{alg:GFM+} with above parameter settings could output a $(\delta,\epsilon)$-Goldstein stationary point of~$f(\,\cdot\,)$ in expectation, and the total stochastic zeroth-order oracle complexity is 
\begin{align*} 
T \left( \left \lceil\frac{b'}{m} \right \rceil+ 2b \right) = \mathcal{O} \left( d^{3/2} \left(  \frac{L^3}{\epsilon^3}+ \frac{\Delta L^2}{\delta \epsilon^3} \right)   \right).
\end{align*}

We formally summarize the result of our analysis as follows.
\begin{thm}[GFM$^+$] \label{thm:GFM+}
Under Assumption \ref{asm:stoc}, \ref{asm:Lip} and \ref{asm:lower}, 
we run GFM${}^+$ (Algorithm~\ref{alg:GFM+}) with
\begin{align*}
&T= \left \lceil \frac{4 \Delta_{\delta}}{\eta \epsilon^2} \right \rceil, ~b' = \left \lceil  \frac{2 \sigma_{\delta}^2}{\epsilon^2} \right \rceil,~b = \left \lceil \frac{2b'}{m}\right \rceil, \\
&\eta = \frac{\sqrt{b'}}{m M_{\delta}} \qquad  \text{and} \qquad  m = \left\lceil\frac{L_{\delta} \sqrt{b'}}{M_{\delta}}\right\rceil,
\end{align*}
where $L_{\delta},\sigma_{\delta},M_{\delta}$ and $\Delta_{\delta}$ follows the definition in (\ref{dfn:const}).
Then it outputs a $(\delta,\epsilon)$-Goldstein stationary point of $f(\cdot)$ in expectation and the total stochastic zeroth-order oracle complexity is at most
\begin{align} \label{eq:final-cmp}
\mathcal{O} \left( d^{3/2} \left( \frac{L^3}{\epsilon^3}+ \frac{\Delta L^2}{\delta \epsilon^3}  \right)   \right),
\end{align}
where $\Delta=f(x_0)-f^*$.
\end{thm}

Theorem \ref{thm:GFM+} achieves the best known stochastic zeroth-order oracle complexity for finding a $(\delta,\epsilon)$-Goldstein stationary point of nonsmooth nonconvex stochastic optimization problem, which improves upon GFM \citep{lin2022gradient} in the dependency of both $\epsilon$ and $L$.


\subsection{The Results for Convex Optimization} \label{sec:convex}

This subsection extends the idea of GFM$^+$ to find  $(\delta,\epsilon)$-Goldstein stationary points for nonsmooth convex optimization problem.
We propose warm-started GFM$^+$ (WS-GFM$^+$) in Algorithm \ref{alg:2GFM+}, which initializes GFM$^+$ with GFM.

The complexity analysis of WS-GFM$^+$ requires the following assumption. 
\begin{asm} \label{asm:c}
We suppose that objective $f:\BR^d\to\BR$ is convex and the set $\fX^* := \arg \min_{x \in \BR^d} f(x)$ is non-empty.
\end{asm}

We remark the assumption of non-empty $\fX^*$ is stronger than Assumption \ref{asm:lower} since the Lipschitzness of $f$ implies  $f(x_0) - \inf_{x \in \BR^d} f(x) \le L\,{\rm dist}(x_0, \fX^*)$.
Therefore, under Assumption \ref{asm:c}, directly using GFM$^+$ (Algorithm~\ref{alg:GFM+}) requires 
\begin{align}\label{complexity:GFM+-cvx}
    \mathcal{O} \left( d^{3/2} \left( \frac{L^3}{\epsilon^3}+ \frac{ L^3 R}{\delta  \epsilon^3}  \right)   \right)
\end{align}
iterations to find a $(\delta,\epsilon)$-Goldstein stationary point of $f$, where we denote~$R = {\rm dist}(x_0, \fX^*)$.

\begin{algorithm}[t]  
\caption{WS-GFM${}^+$($x_0,\eta_0,T_0,\eta,T,m,b,b'$)} \label{alg:2GFM+}
\begin{algorithmic}[1] 
\STATE $x_1= \text{GFM}(x_0,\eta_0,T_0 )$ \\[0.15cm]
\STATE $x_T = \text{GFM}{}^+(x_1,\eta,T,m,b,b')$ \\[0.15cm]
\STATE \textbf{return} $x_T$
\end{algorithmic}
\end{algorithm}

\begin{algorithm}[t]  
\caption{WS-GFM ($x_0,\eta_0,T_0, \eta, T $)} \label{alg:2GFM}
\begin{algorithmic}[1] 
\STATE $x_1= \text{GFM}(x_0,\eta_0,T_0 )$ \\[0.15cm]
\STATE $x_T = \text{GFM}(x_1,\eta,T)$ \\[0.15cm]
\STATE \textbf{return} $x_T$ 
\end{algorithmic}
\end{algorithm}

Next we show the warm-start strategy can improve the complexity in (\ref{complexity:GFM+-cvx}).
It is based on the fact that GFM obtains the optimal stochastic zeroth-order oracle complexity in the measure of function value~\citep{shamir2017optimal}. We include the proof in the appendix for completeness.
\begin{thm}[GFM in function value] \label{thm:upper-c}
Under Assumption \ref{asm:stoc}, \ref{asm:Lip} and \ref{asm:c}, if we run GFM (Algorithm \ref{alg:GFM}) with
$T = \left\lceil {2 \sigma_{\delta} R^2}/{\zeta} \right \rceil$, $\eta = R/\big(\sigma_{\delta} \sqrt{T}\big)$ and
$\delta = \zeta/(4L)$
where $\sigma_{\delta}$ follows definition in (\ref{dfn:const}),
then the output satisfies $\BE[f(x_{\rm out}) - f^* ] \le \zeta$ and the total stochastic zeroth-order oracle complexity is at most
$\fO(d  L^2 R^2 \zeta^{-2})$, where~$R= {\rm dist} (x_0, \fX^*)$.
\end{thm}

Theorem \ref{thm:upper-c} means using the output of GFM as the initialization for GFM$^+$ can make the term $\Delta$ in (\ref{eq:final-cmp}) be small. We denote $\zeta = f(x_1) - f^*$, then the total complexity in WS-GFM$^+$ is
\begin{align*}
    \fO \left( \frac{d^{3/2} L^3}{\epsilon^3} + \frac{d^{3/2} L^3 \zeta }{\delta \epsilon^3} + \frac{d  L^2 R^2}{\zeta^2} \right).
\end{align*}
Then an appropriate choice of $\zeta$ leads to the following result.

\begin{thm}[WS-GFM${}^+$] \label{thm:2GFM+}
Under Assumption \ref{asm:stoc}, \ref{asm:Lip} and \ref{asm:c}, 
Algorithm  \ref{alg:2GFM+} with an appropriate parameter setting can output  a  $(\delta,\epsilon)$-Goldstein stationary point in expectation within the stochastic zeroth-order oracle complexity of 
\begin{align*}
\mathcal{O}\left( \frac{d^{3/2} L^3}{\epsilon^3} + \frac{d^{4/3} L^2 R^{2/3}}{\delta^{2/3} \epsilon^2} \right),
\end{align*}
where $R= {\rm dist} (x_0, \fX^*)$.
\end{thm}

Naturally, we can also use the idea of warm-start to obtain the complexity of GFM~\cite{lin2022gradient} for convex case. We present warm-started GFM (WS-GFM) in Algorithm \ref{alg:2GFM} and provide its theoretical guarantee as follows.

\begin{thm}[WS-GFM] \label{thm:2GFM}
Under Assumption \ref{asm:stoc}, \ref{asm:Lip} and \ref{asm:c}, 
Algorithm \ref{alg:2GFM} with an appropriate parameter setting can output a  $(\delta,\epsilon)$-Goldstein stationary point in expectation within the stochastic zeroth-order oracle complexity of
\begin{align*}
\mathcal{O}\left( \frac{d^{3/2} L^4}{\epsilon^4} + \frac{d^{4/3} L^{8/3} R^{2/3}}{\delta^{2/3} \epsilon^{8/3}} \right),
\end{align*}
where $R= {\rm dist} (x_0, \fX^*)$.
\end{thm}



\section{Numerical Experiments} \label{sec:exp}

\begin{figure*}[t]\centering
\begin{tabular}{ccc}
\includegraphics[scale=0.3]{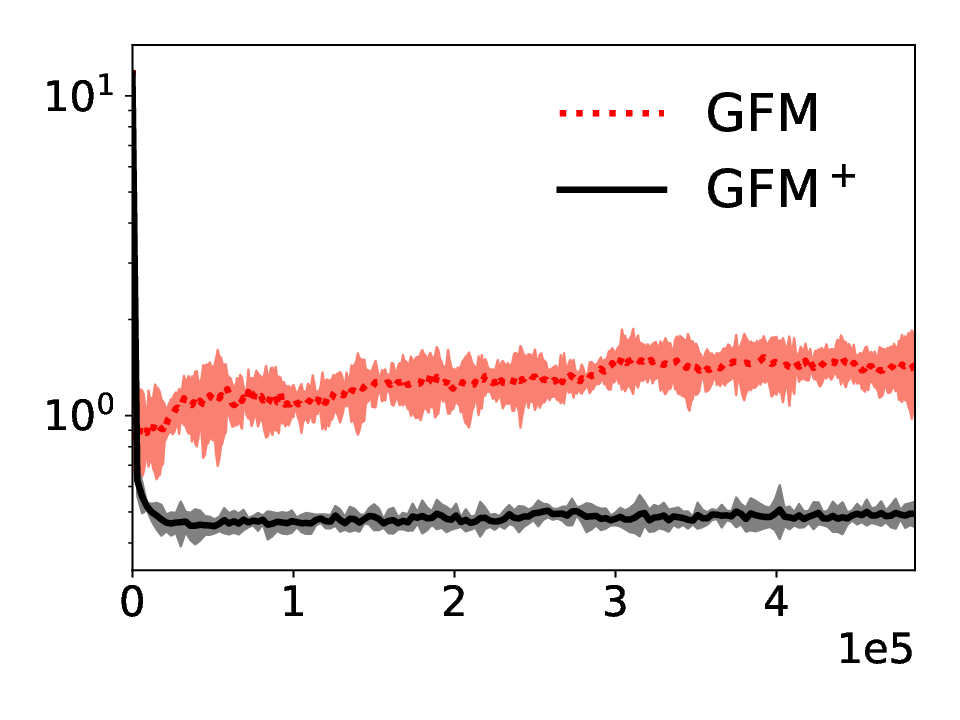} &
\includegraphics[scale=0.3]{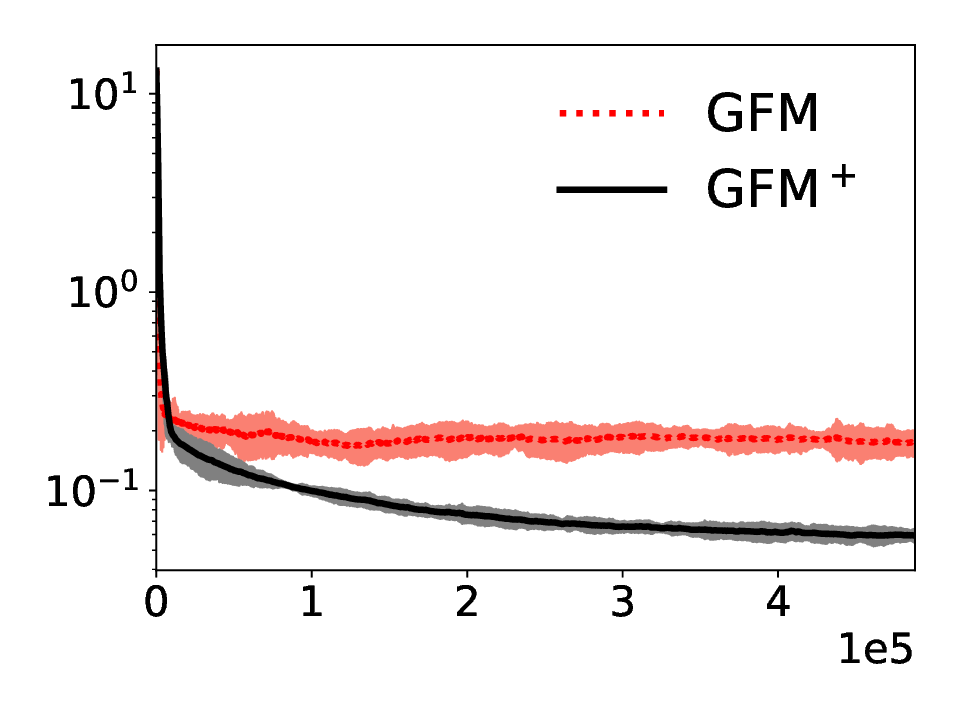} &
\includegraphics[scale=0.3]{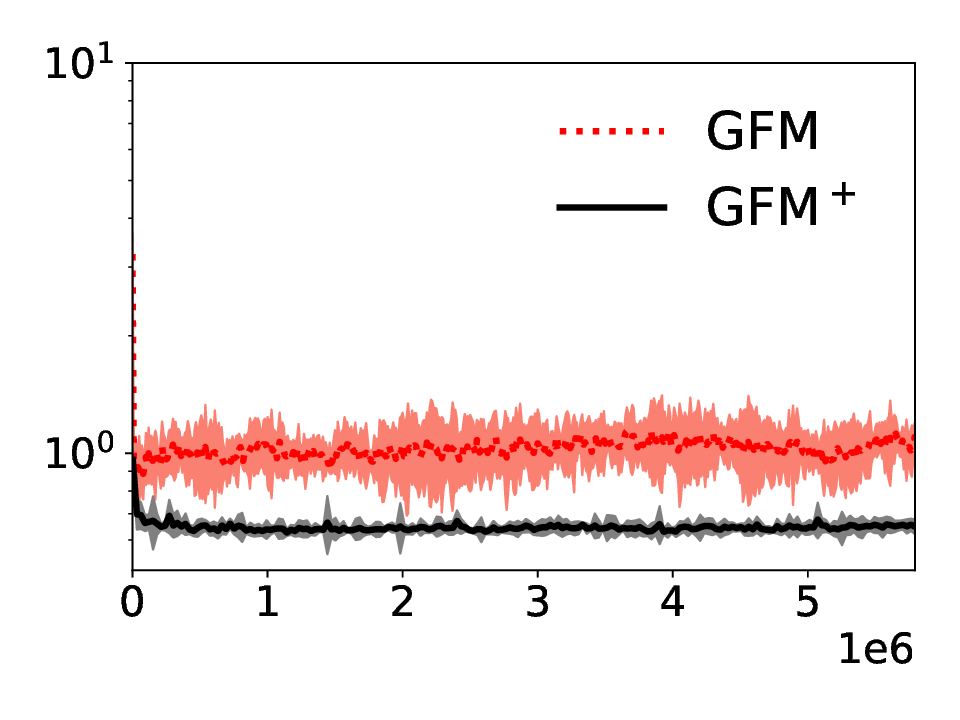} \\
(a) a9a & (b) w8a & (c) covtype \\
\includegraphics[scale=0.3]{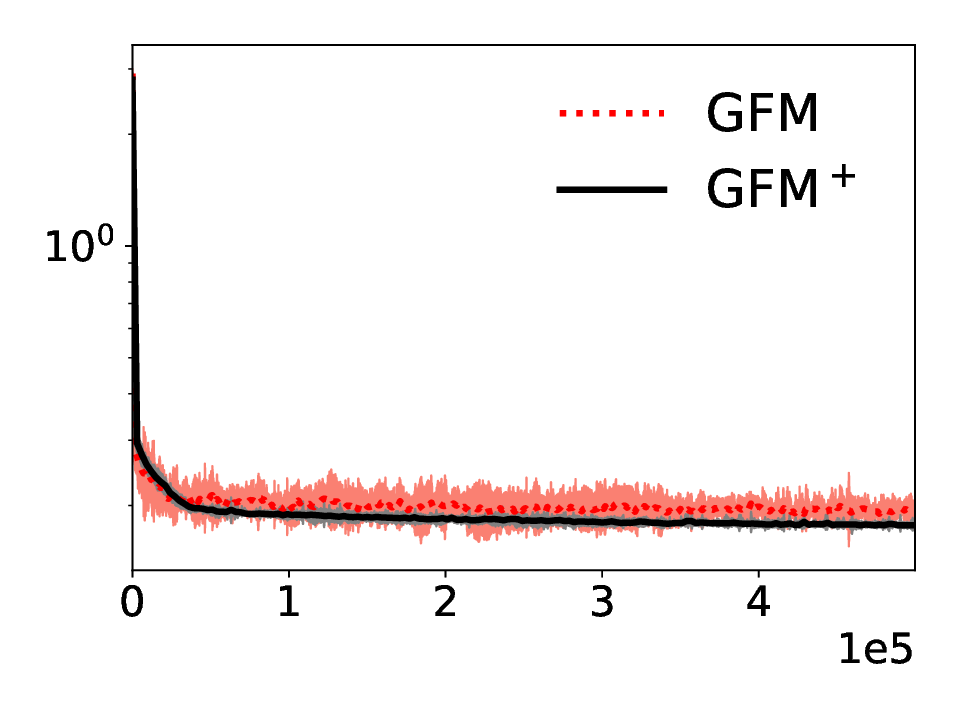} &
\includegraphics[scale=0.3]{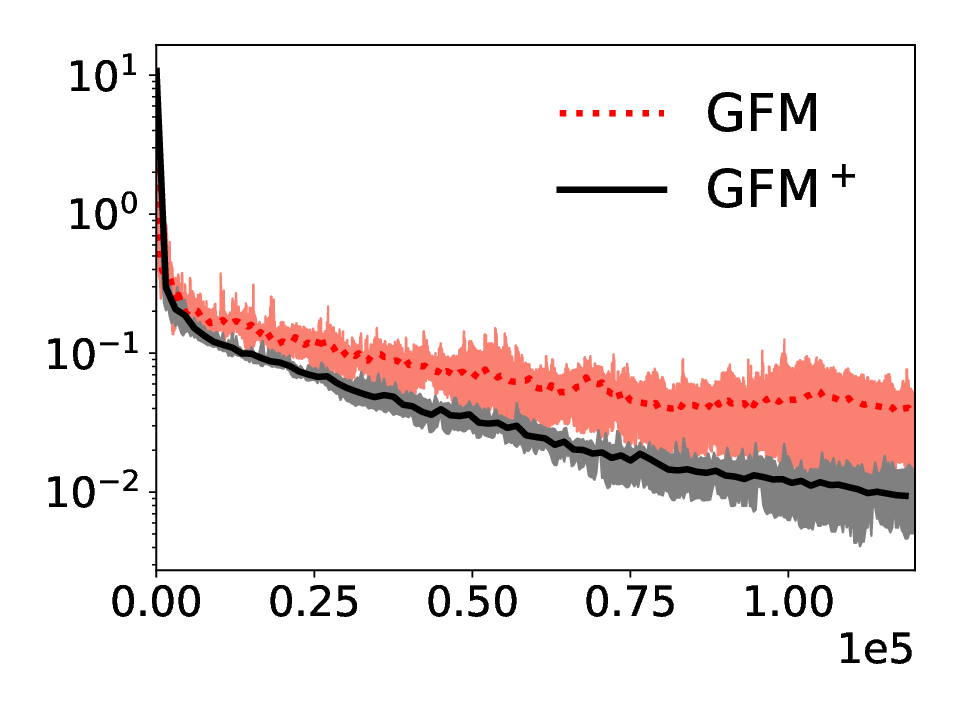} &
\includegraphics[scale=0.3]{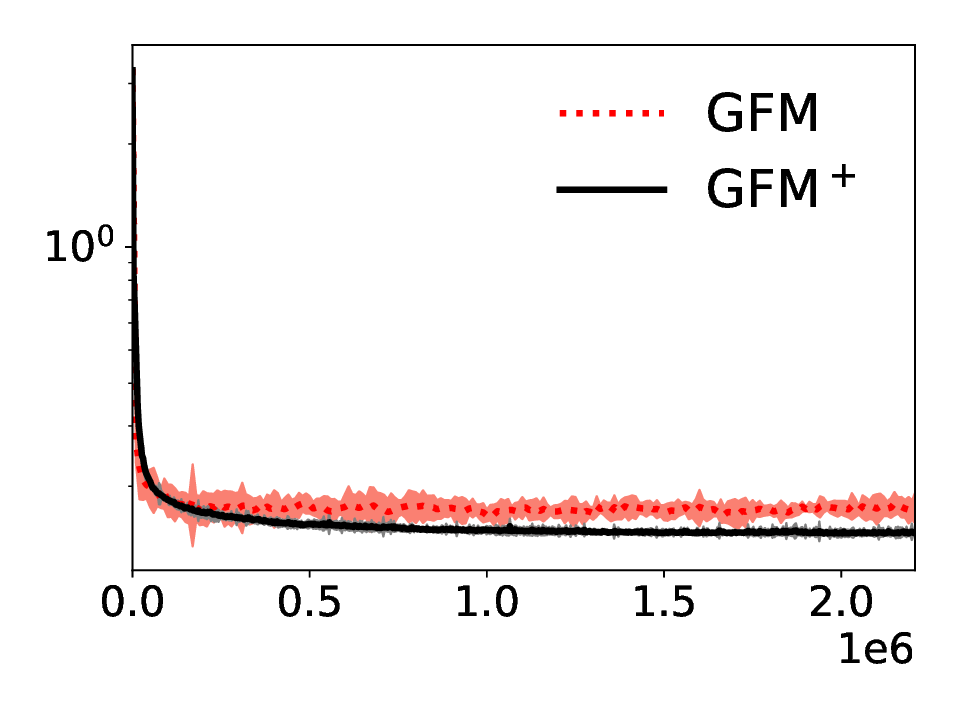} \\
(e) ijcnn1 & (f) mushrooms & (g) phishing \\
\end{tabular}
\caption{For nonconvex penalized SVM, we present the  results for \textbf{complexity  \emph{vs.}loss} on datasets ``a9a'', ``w8a'', ``covtype'', ``ijcnn1'', ``mushrooms'' and ``phishing''. The result for each method is averaged over 20 independent runs.
}\vskip-0.08cm
\label{fig:exp}
\end{figure*}

In this section, we conduct the numerical experiments on nonconvex penalized SVM and black-box attack to show the empirical superiority of proposed GFM$^+$.

\subsection{Nonconvex Penalized SVM}

We consider the nonconvex penalized
SVM with capped-$\ell_1$ regularizer~\citep{zhang2010analysis}.
The model targets to train the binary classifier $x\in\BR^d$ on dataset~$\{(a_i, b_i)\}_{i=1}^{n}$, where~$a_i\in\BR^d$ and $b_i \in \{1,-1 \}$ are the feature of the~$i$-th sample and its corresponding label.
It is formulated as the following nonsmooth nonconvex problem
\begin{align*}
    \min_{x\in\BR^d} f(x) \triangleq \frac{1}{n} \sum_{i=1}^n l(b_i a_i^\top x) +  r(x), 
\end{align*}
where $l(x)=\max\{1-x,0\}$, $r(x) = \lambda \sum_{j=1}^d\!\min \{ \vert x_j \vert,\alpha\}$ and $\lambda, \alpha>0$ are hyperparamters.
We take $\lambda = 10^{-5} / n$ and $\alpha =2$ in our experiments.

We compare the proposed GFM$^+$ with GFM \citep{lin2022gradient} on LIBSVM datasets ``a9a'', ``w8a'', ``covtype'', ``ijcnn1'', ``mushrooms'' and ``phishing''~\cite{chang2011libsvm}.
We set $\delta = 0.001$ and tune the stepsize $\eta$ from~$\{ 0.1,0.01,0.001\}$ for the two algorithms. For GFM$^+$, we tune both $m$ and $b$ in $\{1,10,100 \}$ and set $b' = mb$ by following our theory.
We run the algorithms with twenty different random seeds for each dataset and demonstrate the results in Figure \ref{fig:exp}, where the vertical axis represents the mean of loss calculated by the twenty runs and the horizontal axis represents the zeroth-order complexity.
It can be seen that GFM$^+$ leads to faster convergence and improve the stability in the training.


\subsection{Black-Box Attack on CNN}

\begin{figure*}[t]
    \centering
    \begin{tabular}{c c}
    \qquad\includegraphics[scale=0.36]{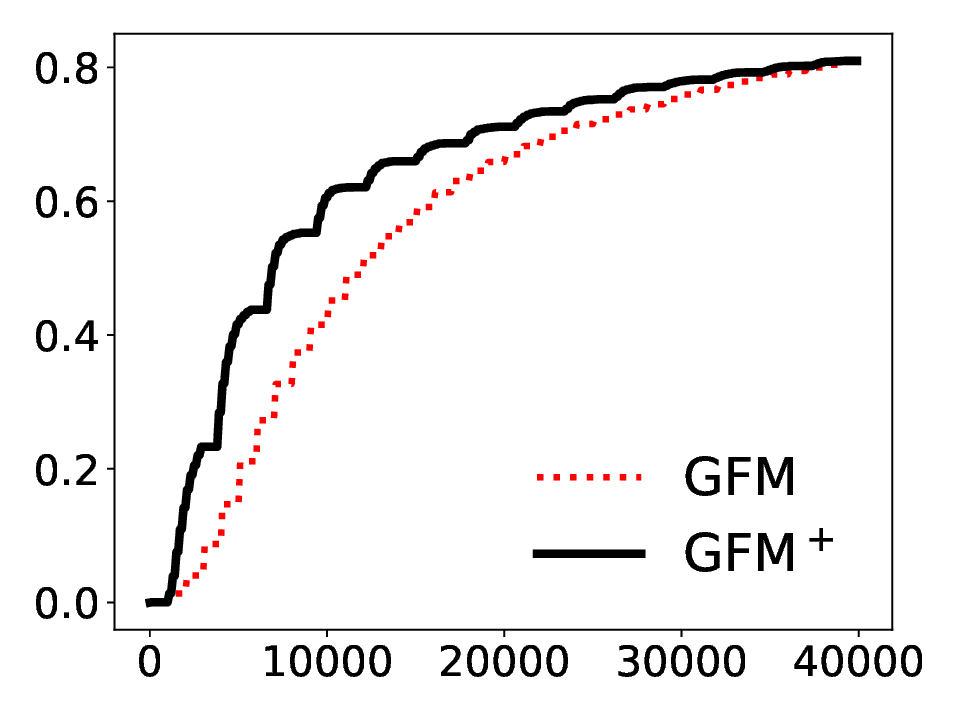}\qquad    &
    \qquad\includegraphics[scale=0.36]{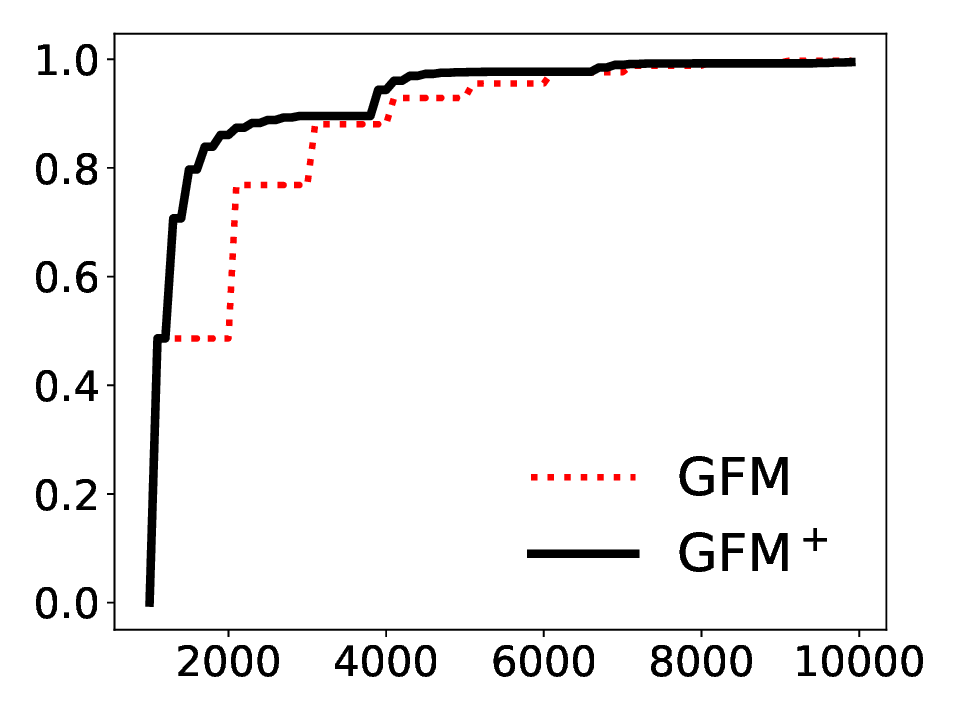}\qquad \\
    (a) MNIST     & (b) Fashion-MNIST 
    \end{tabular}
    \caption{For black-box attack, we present the results for \textbf{complexity \emph{vs.} success rate} on datasets ``MNIST'' and ``Fashion-MNIST''.}
    \label{fig:attack}\vskip-0.07cm
\end{figure*}

We consider the untargeted black-box adversarial attack on image classification with convolutional neural network (CNN).
For a given sample $z \in\BR^d$, we aim to find a sample $x\in\BR^d$ that is close to $z$ and leads to misclassification.
We formulate the problem as the following nonsmooth nonconvex problem \citep{carlini2017towards}:
\begin{align*}
\begin{split}
        \min_{\Vert x - z \Vert_{\infty} \le \kappa} \max\{ \log [Z(x)]_t -  \max_{i \ne t }\log [Z(x)]_i, - \theta  \},
\end{split}
\end{align*}
where  {$\kappa$ is the constraint level of the distortion}, $t$ is the class of sample $z$ and $Z(x)$ is the logit layer representation after softmax in the CNN for $x$ such that~$[Z(x)]_i$ represents the predicted probability that $x$ belongs to class~$i$.
We set~$\theta = 4$ and $\kappa =0.2$ for our experiments.
To address the constraint in the problem, we heuristically conduct  additional projection step for the update of $x$ in GFM and GFM$^{+}$.

We compare the proposed GFM$^+$ with GFM \citep{lin2022gradient} on datasets MNIST and Fashion-MNIST. We use a  LeNet \citep{lecun1998gradient} type CNN which consists of two 5x5 convolution layers two fully-connected layers; 
each convolution layer is associated with max pooling and ReLU
activation;
each fully-connected layer is associated with ReLU activation. The detail of the network is presented in Appendix \ref{apx:exp}.
We train the CNN with SGD by 100 epochs with stepsize starting with 0.1 and decaying by $1/2$ every 20 epochs. 
It achieves classification accuracies $98.95\%$  on the MNIST  and $91.94\%$ on the Fashion-MNIST. 
We use GFM$^+$ and GFM to attack the trained CNNs on all of 10,000 images on testsets and set $\delta =0.01$.
For both GFM and GFM$^+$, we tune $b'$ from $\{500,1000,2000\}$. 
For GFM${}^+$, we additionally tune $m$ from $\{10,20, 50 \}$ and set $b = b'/m$.
For both the two algorithms, we tune the initial learning rate $\eta$ in $\{0.5,0.05,0.005\}$ and decay by $1/2$ if there is no improvement in 10 iterations at one attack.
The experiment results are shown in Figure \ref{fig:attack}, where the vertical axis represents the success attack rate in the whole dataset and the horizontal axis represents the zeroth-order complexity. It shows that the GFM${}^+$ can achieve a better attack success rate than the GFM, within the same number of zeroth-order oracles. 

\section{Conclusion}

This paper proposes the novel zeroth-order algorithm GFM$^{+}$ for stochastic nonsmooth nonconvex optimization.
We prove that the algorithm could find a $(\epsilon,\delta)$-Goldstein stationary point of the Lipschitz objective in expectation within 
at most $\mathcal{O}({ L^3 d^{3/2}}{\epsilon^{-3}}+ \Delta L^2 d^{3/2} \delta^{-1} \epsilon^{-3})$ stochastic zeroth-order oracle complexity, which improves the best known result of GFM~\cite{lin2022gradient}. The numerical experiments on nonconvex penalized SVM and black-box attack on CNN also validate the advantage of GFM$^{+}$. 
However, the tightness for the complexity of GFM$^{+}$ is still unclear. It is interesting to study the zeroth-order and first-order oracle lower bounds for the task of finding approximate Goldstein stationary points. 

\section*{Acknowledgements}

This work is supported by National Natural Science Foundation of China (No. 62206058) and Shanghai Sailing Program (22YF1402900). The authors thank Zhichao Huang for sharing the code for black-box attack on CNN.

\bibliography{reference}
\bibliographystyle{icml2023}

\newpage
\appendix
\onecolumn

\section{The Proof of Proposition \ref{prop:avgL}} 

\begin{proof}
According to the definition of $g(x;w,\xi)$, we use Young's inequality and Proposition \ref{prop:msLip} to obtain
\begin{align*}
    &\quad \BE_{\xi }\Vert g(x;w,\xi) - g(y;w,\xi) \Vert^2 \\
    &\le \frac{d^2}{2 \delta^2} \BE_{\xi }\Vert F(x+ \delta w;\xi) - F(y+\delta w;\xi) \Vert^2 + \frac{d^2}{2 \delta^2} \BE_{\xi} \Vert F(x- \delta w;\xi) - F(y-\delta w;\xi) \Vert^2  \\
    &\le   \frac{d^2L^2}{\delta^2} \Vert x - y \Vert^2.
\end{align*}
\end{proof}

\section{The Tightness of $L_\delta$ and $M_\delta$} \label{apx:tight}

We formally show the tightness of $L_\delta$ and $M_\delta$ as follows.

\begin{prop} \label{tight:L_delta}
The order of Lipschitz constant $L_{\delta}= \fO(\sqrt{d}L \delta^{-1})$ for the gradient of $f_\delta(\,\cdot\,)$ is tight.
\end{prop}
\begin{proof}
According to Lemma 10 of \citet{duchi2012randomized}, it holds that
\begin{align*}
    E_U(f) E_{V}(f) \ge c' L^2 \sqrt{d}
\end{align*}
for some constant $c' >0$, where $E_U(f)$ and $E_V(f) $ are defined as
\begin{align*}
    E_U(f) &= \inf_{\kappa \in \BR} \sup_{x \in \BR^d} \left\{\vert f(x) - f_{\delta}(x) \vert \le \kappa \right\}
\end{align*}
and
\begin{align*}
    E_V(f) &= \inf_{\kappa \in \BR } \sup_{x ,y \in \BR^d} \left\{ \Vert \nabla f_{\delta}(x) - \nabla f_{\delta}(y) \Vert \le \kappa \Vert x - y \Vert\right\}.
\end{align*}
The above lower bound can be taken by the function
\begin{align*}
    f(x) = \frac{L}{2} \Vert x \Vert + \frac{L}{2} \left \vert \frac{x^\top y}{\Vert y \Vert^2} - \frac{1}{2} \right\vert.  
\end{align*}
Using Proposition \ref{prop:zo}, we know that $E_U(f) \le L \delta$ and it immediately implies our claim.
\end{proof}

\begin{prop} \label{tight:M_delta}
The order of mean-squared Lipschitz constant $M_{\delta} = \fO(dL\delta^{-1})$ for the stochastic zeroth-order gradient estimator $g(\,\cdot\,;w,\xi)$ is tight.
\end{prop}
\begin{proof}

We construct a function $f(x) = F(x;\xi) = \frac{L}{\sqrt{d}}\sum_{i=1}^d \vert x_i \vert$ for any random vector $\xi$. It is clear that each $F(x;\xi)$ is $L$-Lipschitz for any random index $\xi$ by noting:
\begin{align*}
    \vert F(x;\xi) - F(y;\xi) \vert = \frac{L}{\sqrt{d}} \vert \Vert x \Vert_1 - \Vert y \Vert_1 \vert \le \frac{L}{\sqrt{d}} \Vert x - y \Vert_1 \le L \Vert x - y \Vert_2.
\end{align*}
Choosing $x = \delta \vone$ and $y = -\delta \vone$, we know that $\Vert x - y \Vert_2 = 2 \sqrt{d} \delta$. Hence, we further know that ${\rm (a)} \triangleq M_{\delta} \Vert x- y \Vert_2 = 2 d^{3/2} L$. Next, we calculate the zeroth order gradient for $w = 1 / \sqrt{d} \cdot  \vone$ and verify that $g(x;w,\xi) =  L d \vone$. Similarly, we also have $g(y;w,\xi) =  -L d \vone$. Therefore, ${\rm (b)} \triangleq \Vert g(x;w,\xi) - g(y;w,\xi) \Vert_2 = 2 d^{3/2} L$. Finally, we complete the proof by noting that term (a) is exactly the same as term (b).

\end{proof}

\section{The Proof of Lemma \ref{lem:delem}}
\begin{proof}
We have
\begin{align*}
f_{\delta}(x_{t+1}) & \le f_{\delta}(x_t) + \nabla f_{\delta}(x_t)^\top (x_{t+1} - x_t) + \frac{L_{\delta}}{2} \Vert x_{t+1} - x_t \Vert^2  \\
&=  f_{\delta}(x_t) - \eta \nabla f_{\delta}(x_t)^\top  v_t + \frac{L_{\delta} \eta^2}{2} \Vert v_t \Vert^2  \\
&=  f_{\delta}(x_t) -  \frac{\eta}{2} \Vert \nabla f_{\delta}(x_t) \Vert^2  -\left( \frac{\eta}{2} - \frac{L_{\delta} \eta^2}{2} \right) \Vert v_t \Vert^2 + \frac{\eta}{2} \Vert v_t - \nabla f_{\delta}(x_t) \Vert^2,
\end{align*}
where the first inequality uses the $L_{\delta}$-smoothness of $f_{\delta}$ in Proposition \ref{prop:smooth}; 
the second line follows the update $x_{t+1} = x_t - \eta v_t$; 
the last step uses fact $2a^\top b = \Vert a \Vert^2 + \Vert b \Vert^2 - \Vert a - b \Vert^2$ for any $a,b\in\BR^d$.
\end{proof}

\section{The Proof of Lemma \ref{lem:vrb}}
\begin{proof}
The update of $v_{t+1}$ leads to
\begin{align*}
\BE \Vert v_{t+1} - \nabla f_{\delta}(x_{t+1}) \Vert^2 &= \BE \Vert v_{t} - g(x_t;S) + g(x_{t+1};S) - \nabla f_{\delta} (x_{t+1}) \Vert^2 \\
&= \BE \Vert v_t - \nabla f_{\delta}(x_t) \Vert^2 + \BE \Vert g(x_{t+1};S) - g(x_t;S) + \nabla f_{\delta}(x_t) - \nabla f_{\delta} (x_{t+1}) \Vert^2 \\
&\le \BE \Vert v_t - \nabla f_{\delta}(x_t) \Vert^2 + \frac{1}{b} \BE \Vert g(x_{t+1};w_1,\xi_1) - g(x_t;w_1,\xi_1) \Vert^2 \\
&\le \BE \Vert v_t - \nabla f_{\delta}(x_t) \Vert^2 + \frac{\eta^2 M_{\delta}^2}{b} \BE \Vert v_t \Vert^2,
\end{align*}
where the first line follows the update rule; 
the second line use the property of martingale~\citep[Proposition 1]{fang2018spider}; 
the first inequality uses the fact $\BE \left \Vert \sum_{i=1}^m a_i \right \Vert^2 = m \BE \Vert a_1 \Vert^2$ for any i.i.d random vector $a_1,a_2\cdots,a_m \in \BR^d$ with zero-mean and Definition \ref{dfn:bgrad}; the second inequality is obtained by Proposition \ref{prop:avgL} and the update $x_{t+1} = x_t - \eta v_t$.
\end{proof}

\section{The Proof of Lemma \ref{lem:epoch} and Corollary \ref{cor:epoch}}
\begin{proof}
We just need to plug the result of Corollary \ref{cor:vrb} into Lemma \ref{lem:delem}.
\end{proof}

\section{The Proof of Theorem \ref{thm:GFM+}}
\begin{proof}
The parameter settings of this theorem guarantee that the term $(*)$ in Corollary \ref{cor:epoch} be positive and we can drop it to obtain inequalities (\ref{eq:main}) and (\ref{eq:station}). 
Using the Jensen's inequality  $\BE \Vert \nabla f_{\delta}(x_{\rm out}) \Vert \le \sqrt{ \BE \Vert \nabla f_{\delta} (x_{\rm out}) \Vert^2 }$ and relationship $\nabla f_{\delta}(x_{\rm out}) \in \partial_{\delta}f (x_{\rm out})$ shown in Proposition \ref{prop:smooth}, we know that the output satisfies
$\BE  {\rm dist}(0, \partial_{\delta} f(x))  \leq \epsilon$.
The overall stochastic zeroth-order oracle complexity is obtained by plugging definitions in (\ref{dfn:const}) into $T \left( \left \lceil b'/m \right \rceil+ 2b \right)$.
\end{proof}

\section{The Proof of Theorem \ref{thm:upper-c}}

\begin{proof}
For any $x^* \in \fX^*$, it holds that
\begin{align*}
&\quad  \mathbb{E}[ f_{\delta}(x_t) - f_{\delta}(x^*)] \\
&\le \mathbb{E} \left[ \nabla f_{\delta}(x_t)^\top (x_t - x^*)\right] \\
&= \frac{1}{\eta} \mathbb{E} \left[ (x_t - x_{t+1} )^\top (x_t - x^*)  \right] \\
&= \frac{1}{2\eta} \mathbb{E}\left[ \Vert x_t - x^* \Vert^2- \Vert x_{t+1} - x^* \Vert^2 + \Vert x_{t+1} - x_t \Vert^2 \right]  \\
&\le \frac{1}{2\eta} \mathbb{E}\left[ \Vert x_t - x^* \Vert^2- \Vert x_{t+1} - x^* \Vert^2\right] + \frac{\eta \sigma_{\delta}^2}{2 },
\end{align*}
where the first inequality uses the convexity of $f_{\delta}(\,\cdot\,)$  and the second inequality uses Proposition \ref{prop:zo}. 
Combining the Jensen's inequality, we have
\begin{align} \label{eq:ws}
    \mathbb{E} [f_{\delta} (x_{\rm out}) - f_{\delta}(x^*)] \le \frac{ R^2}{2 \eta T} + \frac{\eta \sigma_{\delta}^2}{2} = \frac{R \sigma_{\delta}}{2 \sqrt{T}}.
\end{align}
Since Proposition \ref{prop:smooth} means $\Vert f - f_{\delta} \Vert_{\infty} \le L \delta$, the results of (\ref{eq:ws}) implies
\begin{align*}
    \BE[f(x_{\rm out}) - f^* ] \le   \mathbb{E} [f_{\delta} (x_{\rm out}) - f_{\delta}(x^*)] + 2 L \delta \le \frac{R \sigma_{\delta}}{2 \sqrt{T}} + 2 L \delta.
\end{align*}
\end{proof}

\section{Proof of Theorem \ref{thm:2GFM+} and \ref{thm:2GFM}}
\begin{proof}
For Algorithm \ref{alg:2GFM}, the complexity is given by
\begin{align*}
    \fO \left( \frac{d^{3/2} L^4}{\epsilon^4} + \frac{d^{3/2} L^4 \zeta }{\delta \epsilon^4} + \frac{d R^2 L^2}{\zeta^2} \right).
\end{align*}
For Algorithm \ref{alg:2GFM+}, the complexity is given by
\begin{align*}
    \fO \left( \frac{d^{3/2} L^3}{\epsilon^3} + \frac{d^{3/2} L^3 \zeta }{\delta \epsilon^3} + \frac{d R^2 L^2}{\zeta^2} \right).
\end{align*}
Then we tune $\zeta$ to achieve the best upper bound 
using the inequality $ 3\sqrt[3]{a^2 b} \le 2a \zeta + b /\zeta^2 $ for any constant $a>0,b>0$.
\end{proof}

\newpage

\section{The Details of Experiments} \label{apx:exp}

The details of the datasets for the experiments of nonconvex penalized SVM is shown in Table \ref{tab:data}.
The network used in the experiments of black box attack is shown in Table \ref{tab:LeNet}.

\begin{table}[t]
    \centering
    \caption{Summary of datasets used in our SVM experiments}    \label{tab:data}\vskip0.1cm
    \begin{tabular}{ccc}
    \hline
    Dataset & $n$ & $d$ \\ 
    \hline 
    a9a     & ~\,48,842 & 123 \\
    w8a     & ~\,64,700 & 300 \\
    covtype & 581,012 & ~54 \\
    ijcnn1 & ~\,49,990 & ~\,22 \\
    mushrooms & ~~\,8,142 & 112 \\
    phishing & ~\,11,055 & ~\,68 \\
    \hline
    \end{tabular}
\end{table}

\begin{table}[t]
    \centering
    \caption{Network used in the experiments}\label{tab:LeNet}\vskip0.1cm
    \begin{tabular}{c c}
    \hline
    Layer & Size   \\
    \hline  
    Convolution     & $5\times5\times16$ \\
    ReLU &  - \\
    Max Pooling     & $2\times2$ \\
    Convolution     & $5\times5\times16$ \\
    ReLU &  - \\
    Max Pooling     & $2\times2$ \\
    Linear & $3136\times128$  \\
    ReLU & - \\
    Linear & $128\times 10$
    \\
    \hline
    \end{tabular}
\end{table}


\end{document}